\documentclass[a4paper,11pt,twoside]{amsart}
\usepackage[left=1.0in,right=1.0in]{geometry}
\usepackage[all]{xy} 
\usepackage{amsmath, amssymb, amsfonts, latexsym, mdwlist, amsthm, amscd}
\usepackage{subfig}
\usepackage{graphicx}
\usepackage{wrapfig}
\newenvironment{dedication}
        {\begin{quotation}\begin{center}\begin{em}}
        {\par\end{em}\end{center}\end{quotation}}

\usepackage[bookmarks, colorlinks, breaklinks, pdftitle={},
pdfauthor={}]{hyperref}
\hypersetup{linkcolor=blue,citecolor=blue,filecolor=black,urlcolor=blue}
\usepackage{CJKutf8}

\usepackage{tikz}
\usetikzlibrary{calc,cd,trees,positioning,arrows,chains,shapes.geometric,%
    decorations.pathreplacing,decorations.pathmorphing,shapes,%
    matrix,shapes.symbols}

\tikzset{
>=stealth',
  punktchain/.style={
    rectangle,
    rounded corners,
    draw=black, thick,
^{}    
    minimum height=3em,
    text centered,
    on chain},
  line/.style={draw, thick, <-},
  element/.style={
    tape,
    top color=white,
    bottom color=blue!50!black!60!,
    minimum width=8em,
    draw=blue!40!black!90, very thick,
    text width=10em,
    minimum height=3.5em,
    text centered,
    on chain},
  every join/.style={->, thick,shorten >=1pt},
  decoration={brace},
  tuborg/.style={decorate},
  tubnode/.style={midway, right=2pt},
}

\usepackage{paralist}
\setdefaultenum{(a)}{(i)}{}{}
\usepackage{enumitem} 

\usepackage{leftidx}
\usepackage{verbatim}

\def\cA{\mathcal{A}}
\def\cB{\mathcal{B}}

\def\cD{\mathcal{D}}
\def\cE{\mathcal{E}}

\def\cK{\mathcal{K}}
\def\cL{\mathcal{L}}
\def\cM{\mathcal{M}}

\def\cO{\mathcal{O}}

\def\cQ{\mathcal{Q}}

\def\cT{\mathcal{T}}

\newcommand{\mor}[1][]{\xrightarrow{#1}}
\newcommand{\isomor}{\mor[\sim]}

\def\Z{\ensuremath{\mathbb{Z}}}

\def\dim{\mathop{\mathrm{dim}}\nolimits}

\def\Ext{\mathop{\mathrm{Ext}}\nolimits}

\def\Hom{\mathop{\mathrm{Hom}}\nolimits}

\def\RHom{\mathop{\mathbf{R}\mathrm{Hom}}\nolimits}

\def\Id{\mathop{\mathrm{Id}}\nolimits}

\def\Num{\mathop{\mathrm{Num}}\nolimits}

\def\MG13{\ensuremath{{\mathcal M}_{\Gamma_1(3)}}}
\def\tildeMG13{\ensuremath{\widetilde{\mathcal M}_{\Gamma_1(3)}}}

\def\into{\ensuremath{\hookrightarrow}}

\def\Db{\mathrm{D}^{\mathrm{b}}}

\def\abs#1{\left\lvert#1\right\rvert}

\makeatletter
\newtheorem*{rep@theorem}{\rep@title}
\newcommand{\newreptheorem}[2]{%
\newenvironment{rep#1}[1]{%
 \def\rep@title{#2 \ref{##1}}%
 \begin{rep@theorem}}%
 {\end{rep@theorem}}}
\makeatother

\newtheorem{Thm}{Theorem}[section]
\newreptheorem{Thm}{Theorem}
\newtheorem{Prop}[Thm]{Proposition}

\newtheorem{Lem}[Thm]{Lemma}

\newtheorem{Cor}[Thm]{Corollary}
\newreptheorem{Cor}{Corollary}

\newreptheorem{Con}{Conjecture}

\newtheorem{thm-int}{Theorem}

\newtheorem{ThmInt}{Theorem}[section]

\theoremstyle{definition}
\newtheorem{Def-s}[Thm]{Definition}
\newtheorem{Def}[Thm]{Definition}
\newtheorem{Set}[Thm]{Setup}
\newtheorem{Rem}[Thm]{Remark}

\newtheorem{Ex}[Thm]{Example}

\numberwithin{equation}{section}

\def\C{\ensuremath{\mathbb{C}}}

\def\Z{\ensuremath{\mathbb{Z}}}

\def\AA{\ensuremath{\mathcal A}}

\def\EE{\ensuremath{\mathcal E}}

\def\LL{\ensuremath{\mathcal L}}

\def\OO{\ensuremath{\mathcal O}}

\newcommand{\Ku}[1][]{\mathcal{K}u({#1})}

\def\llra{\hbox to 10mm{\rightarrowfill}}
\def\lllra{\hbox to 15mm{\rightarrowfill}}

\def\llla{\hbox to 10mm{\leftarrowfill}}
\def\lllla{\hbox to 15mm{\leftarrowfill}}

\def\K{\mathbb K}

\DeclareMathOperator{\Cone}{Cone}

\def\llra{\hbox to 10mm{\rightarrowfill}}
\def\lllra{\hbox to 15mm{\rightarrowfill}}
\def\Ku{\mathcal{K}\!u}

\newcommand{\ignore}[1]{}

\begin{document}

\title[A refined Derived Torelli Theorem for Enriques surfaces, II]{A Refined Derived Torelli Theorem\\ for Enriques surfaces, II: the non-generic case}

\author[C.~Li, P.~Stellari, and X.~Zhao]{Chunyi Li, Paolo Stellari, and Xiaolei Zhao}

\address{C.L.: Mathematics Institute (WMI), University of Warwick, Coventry,
CV4 7AL, United Kingdom.}
\email{C.Li.25@warwick.ac.uk}
\urladdr{\url{https://sites.google.com/site/chunyili0401/}}

\address{P.S.: Dipartimento di Matematica ``F.~Enriques'', Universit{\`a} degli Studi di Milano, Via Cesare Saldini 50, 20133 Milano, Italy.}
\email{paolo.stellari@unimi.it}
\urladdr{\url{https://sites.unimi.it/stellari}}

\address{X.Z.: Department of Mathematics,
University of California, Santa Barbara,
6705 South Hall
Santa Barbara, CA 93106, USA.}
\email{xlzhao@math.ucsb.edu}
\urladdr{\url{https://sites.google.com/site/xiaoleizhaoswebsite/}}

\thanks{C.~L.~ is a University Research Fellow supported by the Royal Society URF$\backslash$R1$\backslash$201129 “Stability condition and application in algebraic geometry”.
P.~S.~was partially supported by the ERC Consolidator Grant ERC-2017-CoG-771507-StabCondEn, by the research projects PRIN 2017 ``Moduli and Lie Theory'' and FARE 2018 HighCaSt (grant number R18YA3ESPJ). X.~Z.~ was partially supported by the Simons Collaborative Grant 636187.}



\begin{abstract}
We prove that two Enriques surfaces defined over an algebraically closed field of characteristic different from $2$ are isomorphic if their Kuznetsov components are equivalent. This improves and completes our previous result joint with Nuer where the same statement is proved for generic Enriques surfaces.
\end{abstract}

\maketitle

\begin{dedication}
\hfill {\small Ad Olivier Debarre con amicizia ed ammirazione.}

\hfill{\small
\begin{CJK*}{UTF8}{gbsn}
谨以此文恭祝Olivier Debarre教授诸事顺遂\; 友谊万古青\; 桃李满天下
\end{CJK*}}
\end{dedication}

\setcounter{tocdepth}{1}
\tableofcontents

\section*{Introduction}

The bounded derived category of coherent sheaves $\Db(X)$ of an Enriques surface $X$ has been widely investigated. Over the complex numbers, a very nice result by Bridgeland and Maciocia \cite{BM01} shows that the derived category determines the surface up to isomorphism. This goes under the name of Derived Torelli Theorem as it can be viewed as a categorification of the usual Hodge-theoretic Torelli Theorem for Enriques surfaces.

In this paper, we want to work over any algebraically closed field $\K$ of characteristic different from $2$ and, under this assumption, Enriques surfaces have a uniform definition: they are smooth projective surfaces $X$ with $2$-torsion dualizing sheaf and such that $H^1(X,\OO_X)=0$. From the geometric point of view, they are quotients of a K3 surface by a fixed-point-free involution. It is a fact \cite{HLT17} that the Derived Torelli Theorem above holds in any characteristic different from $2$ as soon as the field is algebraically closed.

It was first showed in \cite{Zu} that $\Db(X)$ can be further decomposed into interesting pieces. As we will explain in Section \ref{subsect:Enriques}, such a triangulated category always has a semiorthogonal decomposition
\[
\Db(X)=\langle\Ku(X,\cL),\cL\rangle,
\]
where $\cL=\langle L_1,\dots,L_{10}\rangle$ is generated by $10$ line bundles. If $X$ is generic in moduli or generic in the divisor of the moduli space of Enriques surfaces parametrizing nodal Enriques surfaces (i.e.\ containing $(-2)$-curves), then the $10$ line bundles are completely orthogonal. Otherwise, as clarified in Proposition \ref{prop:exceptionalcollectionexists}, these line bundles organize themselves in completely orthogonal blocks and inside of each such block they differ by a rational curve. The residual category $\Ku(X,\cL)$ is called the \emph{Kuznetsov component} of $\Db(X)$. Note that, as it is expected and will be explained later (see Corollary \ref{cor:nonintri}), the definition of $\Ku(X,\cL)$ depends on the choice of $\cL$.

The study of the properties of the Kuznetsov component got a great impulse from \cite{KI15} where the authors also investigated its relation to the derived category of generic Artin\textendash Mumford quartic double solids. In our previous paper \cite{LNSZ} we proved that for a generic Enriques surface $X$, its Kuznetsov component determines $X$ up to isomorphism. As this refines the Derived Torelli Theorem mentioned above, we refer to it as the \emph{Refined Derived Torelli Theorem} for Enriques surfaces. A natural question which remained open after \cite{LNSZ} was whether such a result holds in general, that is for every Enriques surface.

In this paper we positively answer this question by means of the following theorem which is the main result of this paper.

\begin{ThmInt}\label{thm:derived_torreli}
Let $X_1$ and $X_2$ be Enriques surfaces over an algebraically closed field $\K$ of characteristic different from $2$. If they possess semiorthogonal decompositions
\begin{equation*}\label{cond2}
\Db(X_i)=\langle \Ku(X_i,\cL_i),\cL_i\rangle,
\end{equation*}
where $\LL_i$ is as above, and there exists an exact equivalence $\mathsf{F}: \Ku(X_1,\cL_1)\xrightarrow{\sim}\Ku(X_2,\cL_2)$ of Fourier\textendash Mukai type, then $X_1\cong X_2$.
\end{ThmInt}

This result sits in a very active research area. Statements as the above one are usually referred to as \emph{categorical Torelli theorems} because, in analogy with the Hodge-theoretic Torelli theorems, they show that the geometry of a smooth projective variety can be reconstructed, up to isomorphism, from one relevant component of its derived category (possibly with some additional data). Such a component is meant to play the role of intermediate Jacobians or middle cohomologies and the additional data is the analogue of principal polarizations or special Hodge structures. In the case of Fano manifolds, many papers were devoted to this kind of problems including \cite{BEA,BLMS,BMMS,HR,LPZ,PY}. All these results heavily rely on the existence of Bridgeland stability conditions on the Kuznetsov component as proved in \cite{BLMS}. In general, we expect that there are no stability conditions on the Kuznetsov components of Enriques surfaces.

For this reason, in this paper we take a different perspective which is indeed the same as in \cite{LNSZ}. Indeed, we prove Theorem \ref{thm:derived_torreli} in Section \ref{subsect:proofthm} by showing that the given equivalence between the Kuznetsov components can be extended to the whole derived categories by adding the $10$ line bundles one by one. Once we know that $\Db(X_1)$ and $\Db(X_2)$ are equivalent, we can invoke the Derived Torelli Theorem and conclude that $X_1$ and $X_2$ are isomorphic. It should be noted that Theorem \ref{thm:derived_torreli} is actually a special instance of the more precise Theorem \ref{thm:gen}.

This extension procedure is made possible by Proposition \ref{prop:extension} which was proved in \cite{LNSZ}. In turns, to make such a result applicable in the more complicated geometric setting of this paper, we need a complete classification of special objects in the Kuznetsov component: $3$-spherical and $3$-pseudoprojective objects. If $X$ is generic and thus the $10$ line bundles are orthogonal, we only get $3$-spherical objects and their classification in \cite{LNSZ} is considerably simpler. The hard part of this paper consists in dealing with $3$-pseudoprojective objects which appear exactly because the line bundles in the semiorthogonal decomposition are no longer orthogonal. This is the content of Theorem \ref{thm:sphericals} which is indeed the technical core of this paper. In order to make the paper self-contained we include in Sections \ref{subsec:sodserre} and \ref{subsect:Enriques} some introductory material about semiorthogonal decompositions and the geometry of Enriques surfaces.

We conclude the introduction by pointing out that in \cite{LNSZ} (see Remark 5.3 therein) we proposed a different approach to prove Theorem \ref{thm:derived_torreli}. The strategy was to deform two Enriques surface $X_1$ and $X_2$ containing $(-2)$-curves and with an equivalence of Fourier\textendash Mukai type $\Ku(X_1,\cL_1)\cong\Ku(X_2,\cL)$ to generic unnodal Enriques surfaces and then apply the generic Refined Derived Torelli Theorem in \cite{LNSZ}. Even though this strategy is still valid, a precise implementation would require a careful (and possibly complicated) comparison between the deformation theory of the Enriques surface and of its Kuznetsov component. We feel like that the approach in the present paper is more direct and conceptually clearer.

\section{Semiorthogonal decompositions and Enriques categories}\label{sect:sodEnriques}

In this section we provide a short introduction to semiorthogonal decompositions of triangulated categories and we apply this to the derived categories of Enriques surfaces. In Section \ref{subsect:useful} we start discussing some preparatory material which will be used later for the classification of special objects in the Kuznetsov component of those surfaces.

\subsection{Semiorthogonal decompositions and Serre functors}\label{subsec:sodserre}

Let $\cT$ be a triangulated category. A \emph{semiorthogonal decomposition} of $\cT$
	\begin{equation*}
	\cT = \langle \cD_1, \dots, \cD_m \rangle
	\end{equation*}
	is a sequence of full triangulated subcategories $\cD_1, \dots, \cD_m$ of $\cT$ such that: 
	\begin{enumerate}
		\item $\Hom(F, G) = 0$, for all $F \in \cD_i$, $G \in \cD_j$ and $i>j$;
		\item For any $F \in \cT$, there is a sequence of morphisms
		\begin{equation*}  
		0 = F_m \to F_{m-1} \to \cdots \to F_1 \to F_0 = F,
		\end{equation*}
		such that $\pi_i(F):=\mathrm{Cone}(F_i \to F_{i-1}) \in \cD_i$ for $1 \leq i \leq m$. 
	\end{enumerate}
The subcategories $\cD_i$ are called the \emph{components} of the decomposition. We say that two such distinct components $\cD_i$ and $\cD_j$ are \emph{orthogonal} if $\Hom_\cT(F,G)\cong\Hom_\cT(G,F)=0$, for all $F$ in $\cD_i$ and $G$ in $\cD_j$.

We will be mostly interested in the case where the inclusion functor $\alpha_i\colon\cD_i\hookrightarrow\cT$  has left adjoint $\alpha_{i}^*$ and right adjoint $\alpha_{i}^!$. In this case, $\cD_i$ is called an \emph{admissible subcategory}.

\begin{Rem}\label{rmk:Serre}
Suppose now that $\cT$ has Serre functor $\mathsf{S}_\cT$ and a semiorthogonal decomposition as above, where all components are admissible subcategories. Then, for any pair of objects $F$ and $G$ in $\cD_i$, we have
\[
\begin{split}
\Hom_{\cD_i}(F,G)&\cong\Hom_{\cT}(\alpha_i(F),\alpha_i(G))\\&\cong\Hom(\alpha_i(G),\mathsf{S}_\cT(\alpha_i(F)))^\vee\\&\cong\Hom(G,\alpha_i^!(\mathsf{S}_\cT(\alpha_i(F))))^\vee.
\end{split}
\]
Thus $\cD_i$ has Serre functor $\mathsf{S}_{\cD_i}$ as well and there is a natural isomorphism $\mathsf{S}_{\cD_i}\cong \alpha_i^!\circ\mathsf{S}_\cT\circ\alpha_i$.
\end{Rem}

We need now to introduce another functor which is naturally defined in presence of a semiorthogonal decomposition satisfying the additional assumption that its components are admissible subcategories. The \emph{left mutation functor} through $\cD_i$ is the functor $\mathsf{L}_{\cD_i}$ defined by the canonical distinguished triangle
\begin{equation}\label{eqn:defofleftmut}
    \alpha_i\alpha_i^!\xrightarrow{\eta_i}\mathsf{id}\rightarrow \mathsf{L}_{\cD_i},
\end{equation}
where $\eta_i$ denotes the counit of the adjunction. In complete analogy, one defines the \emph{right mutation functor} through $\cD_i$ as the functor $\mathsf{R}_{\cD_i}$ defined by the canonical distinguished triangle
\begin{equation}\label{eqn:defofrightmut}
    \mathsf{R}_{\cD_i}\rightarrow\mathsf{id}\xrightarrow{\epsilon_i} \alpha_i\alpha_i^*,
\end{equation}
where $\epsilon_i$ is the unit of the adjunction.

The following is a very simple and well-known result.

\begin{Lem}\label{lem:distriangle}
Let $\cT=\langle \cD_1,\cD_2\rangle$ be a semiorthogonal decomposition. Then  we have an isomorphism of exact functors $\mathsf{L}_{\cD_1}|_{\cD_2}\cong \mathsf{S}_\cT\circ\mathsf{S}^{-1}_{\cD_2}$.
\end{Lem}

\begin{proof}
We have the following natural isomorphisms of exact functors
\begin{equation}\label{eqn:is}
    \mathsf{L}_{\cD_1}|_{\cD_2}\cong \mathsf{L}_{\cD_1}\circ\mathsf{S}^{-1}_\cT\circ\mathsf{S}_\cT|_{\cD_2}\cong \mathsf S^{-1}_{\mathsf{S}_\cT(\cD_2)}\circ\mathsf{S}_\cT|_{\cD_2}\cong \mathsf{S}_\cT\circ\mathsf{S}^{-1}_\cT\circ\mathsf S^{-1}_{\mathsf{S}_\cT(\cD_2)}\circ\mathsf{S}_\cT|_{\cD_2}\cong\mathsf {S}_\cT\circ\mathsf S^{-1}_{\cD_2}.
\end{equation}
Here the non-trivial isomorphism is the second one which follows from  \cite[Lemma 2.6]{KuzCYcat} and the easy observation that
\[
\cT=\langle\mathsf{S}_\cT(\cD_2),\cD_1\rangle
\]
is yet another semiorthogonal decomposition for $\cT$. The last isomorphism in \eqref{eqn:is} follows from an explicit computation using Remark \ref{rmk:Serre}.
\end{proof}

We conclude this section by noting that, from now on, we will assume that all categories are linear over a field $\K$. In this case, $E\in\cT$ is \emph{exceptional} if $\Hom(E,E[p])=0$, for all integers $p\neq0$, and $\Hom(E,E)\cong \K$. A set of objects $\{E_1,\ldots,E_m\}$ in $\cT$ is an \emph{exceptional collection} if $E_i$ is an exceptional object, for all $i$, and $\Hom(E_i,E_j[p])=0$, for all $p$ and all $i>j$.

\subsection{Enriques surfaces}\label{subsect:Enriques}
In the rest of the paper we work over an algebraically closed field $\K$ of characteristic different from $2$. Let $X$ be an Enriques surface. For later use, let us recall that its Serre functor $\mathsf{S}_X$ is defined as $\mathsf S_X(-):=(-)\otimes\omega_X[2]$, where $\omega_X$ is $2$-torsion but non-trivial. Hence $\mathsf S_X^2=[4]$ but $\mathsf{S}_X\ne[2]$. Moreover the torsion part of the N\'eron\textendash Severi group $\mathrm{NS}(X)_{\mathrm{tor}}$ is $2$-torsion as well and generated by the class of $\omega_X$. We set $\Num(X):=\mathrm{NS}(X)/\mathrm{NS}(X)_{\mathrm{tor}}$. For a brief but very informative survey about the geometry of Enriques surfaces, one can have a look at \cite{D}. A more complete treatment is in \cite{CD89,DK}.

Let us now consider the bounded derived category of coherent sheaves $\Db(X)$ of $X$. On one side we have the well-known Derived Torelli Theorem which is a summary of \cite[Proposition 6.1]{BM01} and \cite[Theorem 1.1]{HLT17}.

\begin{Thm}[Bridgeland\textendash Maciocia, Honigs\textendash Lieblich\textendash Tirabassi]\label{thm:dertordercat}
Let $X$ and $Y$ be smooth projective surfaces defined over an algebraically closed field $\K$ of characteristic different from $2$. If $X$ is an Enriques surface and there is an exact equivalence $\Db(X)\cong\Db(Y)$, then $X\cong Y$.
\end{Thm}

On the other hand we can study the structure of $\Db(X)$ by means of the following result (which is mostly well-known to experts).

\begin{Prop}[{\cite[Proposition 3.5]{LNSZ}}]\label{prop:exceptionalcollectionexists}
Let $X$ be an Enriques surface over $\K$. Then $\Db(X)$ contains an admissible subcategory $\cL=\langle\cL_1,\dots,\cL_c\rangle$, where $\cL_1,\dots,\cL_c$ are orthogonal admissible subcategories and
\[
\cL_i=\langle L^i_1,\dots,L^i_{n_i}\rangle
\]
where
\begin{itemize}
\item[{\rm (1)}] $L_j^i$ is a line bundle such that $L^i_j=L^i_1\otimes\OO_X(R^i_1+\dots+R^i_{j-1})$ where $R^i_1,\dots,R^i_{j-1}$ is a chain of $(-2)$-curves of $A_{j-1}$ type; 
\item[{\rm (2)}] $\{L^i_1,\dots,L^i_{n_i}\}$ is an exceptional collection; and
\item[{\rm (3)}] $n_1+\dots+n_c=10$.
\end{itemize}
\end{Prop}

In view of the previous proposition we can state the following.

\begin{Def}\label{def:type}
If $X$ is an Enriques surface with a semiorthogonal decomposition as in Proposition \ref{prop:exceptionalcollectionexists}, then the unordered sequence of positive integers $\{n_1,\dots,n_c\}$ is the \emph{type} of the semiorthogonal decomposition.
\end{Def}

It is worth discussing when the type of the semiorthogonal decomposition may be seen as an invariant associated to $X$.

\begin{Rem}
If $X$ is generic in moduli, then all possible semiorthogonal decompositions as in the statement above get much simplified. Indeed, $X$ does not contain $(-2)$-curves and then one always gets $10$ orthogonal blocks, each consisting of only one line bundle. Thus, for unnodal Enriques surfaces, all semiorthogonal decompositions as above are of the type given by the $10$-uple $\{1,\dots,1\}$.
\end{Rem}

The situation gets much more interesting if one looks at the generic Enriques surfaces in the divisor of the moduli space of such surfaces parametrizing \emph{nodal} Enriques surfaces (i.e.\ surfaces containing $(-2)$-curves). This is very much related to the fact that a natural polarization (called \emph{Fano polarization}) is ample and not just nef. For an extensive discussion on this, the reader can have  a look at \cite[Section 3.1]{LNSZ}. In this paper we will be interested in the following.

\begin{Ex}\label{ex:twoFano}
If $X$ is a generic nodal Enriques surface, then by \cite[Theorem 3.2.2]{Cos} and \cite[Corollary 4.4]{DM}, the surface $X$ has an ample Fano polarization and thus a semiorthogonal decomposition as in Proposition \ref{prop:exceptionalcollectionexists} with $c=10$ and of type given by the $10$-uple $\{1,\dots,1\}$. On the other hand, by Theorem 6.5.4 and Corollary 6.5.9 in \cite{DK}, any generic nodal Enriques surface has a nef Fano polarization which is not ample and gives rise to a semiorthogonal decomposition as in the proposition above but where $c=9$ and of type given by the $9$-uple $\{1,\dots,1,2\}$. This implies that the type of a semiorthogonal decomposition as in Proposition \ref{prop:exceptionalcollectionexists} is not an invariant of the surface.
\end{Ex}

Let us now consider the following slightly more general setting. Let $\cL=\langle\cL_1,\dots,\cL_c\rangle$  be an admissible subcategory of $\Db(X)$, for $X$ an Enriques surface such that $\cL_i=\langle L^i_1,\dots,L^i_{n_i}\rangle$ and
\begin{enumerate}
\item[(a)] the admissible subcategories $\cL_i$'s are orthogonal;
\item[(b)] the objects $L^i_j$'s satisfy properties (1) and (2) in Proposition \ref{prop:exceptionalcollectionexists}\footnote{Note that here we do not assume that condition (3) in Proposition \ref{prop:exceptionalcollectionexists} is satisfied by $\cL$ as in the subsequent argument we will need to make induction on the number of exceptional objects.}.
\end{enumerate}
In this situation, consider the admissible subcategory
\begin{equation*}
\Ku(X,\cL):=\LL^\perp=\langle \cL_1,\dots,\cL_c\rangle^\perp
\end{equation*}
which we will call the \emph{Kuznetsov component} of $X$.
By definition, $\Db(X)$ admits then a semiorthogonal decomposition: 
\begin{equation}\label{eqn:semicond}
\Db(X)=\langle\Ku(X,\cL),\cL\rangle.
\end{equation}

\subsection{Some useful computations}\label{subsect:useful}

In this section we discuss some preliminary computations concerning special objects in $\Db(X)$, for an Enriques surface $X$ with a semiorthogonal decomposition as in \eqref{eqn:semicond} and satisfying properties (a) and (b). We set
\[
\zeta_\cK\colon\Ku(X,\cL)\hookrightarrow\Db(X)
\]
to be the embedding, where $\cK$ is a shorthand for $\Ku(X,\cL)$. For fixed $1\leq i\leq c$ and $1\leq j< l\leq n_i$ we set
\begin{equation}\label{eqn:Qij}
Q^i_{j,l}:=\mathrm{Coker}(L_j^i\hookrightarrow L_l^i).
\end{equation}

\begin{Rem}\label{rmk:Qs}
Note that $Q^i_{j,l}$ is supported on a chain of $(-2)$-curves of $A_{l-j}$-type. This immediately implies that $Q^i_{j,l}\otimes \omega_X\cong Q^i_{j,l}$. In particular, $\langle Q^i_{j,l}\rangle^\perp={}^\perp\langle Q^i_{j,l}\rangle$.
\end{Rem}

The following rather technical result will be used later.

\begin{Lem}\label{lem:compute}
In the above setting, we have that
\begin{enumerate}
    \item [\rm{(1)}] For every object $L$ in $\cL$, $\chi(L,L)\geq 0$;
\item[\rm{(2)}] $\RHom(\mathsf{S}_\cL(L_1^i),L_1^i)\cong\K\oplus \K[-1]\oplus \K[-2]$, when $n_i\geq 2$, while $\RHom(\mathsf{S}_\cL(L_1^i),L_1^i)\cong\K$, when $n_i=1$;
\item[\rm{(3)}]  $\zeta^!_{\cK}(L_1^i)\cong \dots\cong\zeta^!_{\cK}(L^i_{n_i})\cong\zeta^!_{\cK}\circ\mathsf{S}^{-1}_\cL (L_1^i)$.
\end{enumerate}
\end{Lem}
\begin{proof}
Let us first prove (1). Note that $\chi(L^j_i,L^{m}_{l})=\delta_{il}\delta_{jm}$. Indeed, if $j\neq m$, then the objects $L^j_i$ and $L^{m}_{l}$ are orthogonal, for all $i$ and $l$, and thus $\chi(L^j_i,L^{m}_{l})=0$. On the other hand, if $j=m$ and $i\neq l$, then $\chi(L^j_i,L^{m}_{l})=0$ by Proposition \ref{prop:exceptionalcollectionexists} (1) and a simple computation with Riemann\textendash Roch. Finally, if $j=m$ and $i=l$, then $\chi(L^j_i,L^{m}_{l})=1$, because the object is exceptional. This shows that the bilinear form $\chi(-,-)$ is positive definite on the numerical Grothendieck group $K_{\mathrm{num}}(\cL)$ of $\cL$ and thus the claim follows.

As for (2), when $n_i=1$, we have $\mathsf S_\cL(L^i_1)=\mathsf S_{\cL_i}(L^i_1)=L^i_1$. As $L^i_1$ is exceptional, the statement holds by definition.

When $n_i\geq 2$, we first observe that we have the following isomorphisms
\begin{equation}\label{eqn:Serreright}
\mathsf{S}^{-1}_\cL(L^i_1)\cong\mathsf S^{-1}_{\cL_i}(L^i_1)\cong\mathsf{R}_{\langle L_2^i,\dots,L_{n_i}^i\rangle}(\mathsf{S}_{\cL_i}\circ\mathsf{S}^{-1}_{\cL_i}(L_1^i))\cong\mathsf{R}_{\langle L_2^i,\dots,L_{n_i}^i\rangle}(L_1^i)
\end{equation}
where the first one follows from the fact that the admissible subcategories $\cL_i$'s are orthogonal and thus $\mathsf S^{-1}_{\cL_i}(L^i_1)=\mathsf S^{-1}_\cL (L^i_1)$. The second isomorphism is indeed \cite[Lemma 2.7]{KuzCYcat} applied to the object $\mathsf S^{-1}_{\cL_i} (L^i_1)$ in $\cB$ with the following assignments
\[
\cA = \langle L_2^i,\dots, L^i_{n_i}\rangle\qquad  \cB = \langle \mathsf S^{-1}_{\cL_i} (L^i_1)\rangle\qquad\cT = \langle \cA,\cB\rangle=\cL_i,
\]
and noticing that $\mathsf S_{\cB}=\Id_{\cB}$.

Pick $j\in\{1,\dots,n_i-1\}$. Note that $\RHom(L^i_j,L_{j+1}^i)\cong H^*(X,\cO_{X}(R^i_j))\cong\K\oplus \K[-1]$ by property (b) of the semiorthogonal decomposition \eqref{eqn:semicond}. Therefore, we have
the distinguished triangle
\[
\mathsf R_{L_{j+1}^i}(L_j^i)\rightarrow L_j^i\rightarrow L_{j+1}^i\oplus L_{j+1}^i[1].
\]
Consider the commutative diagram where the rows are distinguished triangles
\begin{center}
	\begin{tikzcd}
	L_{j+1}^i \arrow{d} \ar[equal]{r}
		& L_{j+1}^i \arrow{r} \arrow{d}{} & 0 \arrow{d}{}\arrow{r} &  L_{j+1}^i[1] \arrow{d}{}\\
			L_{j+1}^i[-1]\oplus L_{j+1}^i \arrow{r}
		& \mathsf{R}_{L_{j+1}^i}(L_j^i) \ar{r}  & L_j^i\arrow{r}& L_{j+1}^i\oplus L_{j+1}^i[1].
	\end{tikzcd}
\end{center}
By the octahedron axiom, we get the morphisms
\[
\Cone(L_{j+1}^i\rightarrow \mathsf{R}_{L_{j+1}^i}(L_j^i))\cong \Cone(L_j^i[-1]\rightarrow L_{j+1}^i[-1])\cong Q^i_{j,j+1}[-1].
\]
In particular, we have the distinguished triangle
\begin{equation}\label{eq:rlq}
L^i_{j+1}\rightarrow \mathsf R_{L_{j+1}^i}(L_j^i)\rightarrow Q^i_{j,j+1}[-1]
\end{equation}
for every $1\leq j\leq n_i-1$. If we apply $\mathsf R_{\langle L^i_{j+2},\dots,L^i_{n_i}\rangle}$ to \eqref{eq:rlq} and we take into account that  $Q^i_{j,j+1}\in (L_l^i)^\perp \cap \!^\perp\! L_l^i$ for all $l\neq j, j+1$, then we get the distinguished triangle
\begin{equation}\label{eq:rlrlq}
    \mathsf R_{\langle L^i_{j+2},\dots,L^i_{n_i}\rangle}(L^i_{j+1})\rightarrow \mathsf R_{\langle L^i_{j+1},\dots,L^i_{n_i}\rangle}(L_j^i)\rightarrow Q^i_{j,j+1}[-1],
\end{equation}
for every $1\leq j\leq n_i-1$. If we apply the functor $\RHom(L^i_1,-)$ to \eqref{eq:rlrlq} when $j\geq 2$ (which means that $n_i\geq3$), then we get the isomorphisms of graded vector spaces
\begin{align}\label{eqn:isograd}
\RHom(L^i_1, \mathsf R_{\langle L^i_{3},\dots,L^i_{n_i}\rangle}(L^i_{2}))&\cong\RHom(L_1^i, \mathsf R_{\langle L^i_{4},\dots,L^i_{n_i}\rangle}(L^i_{3}))\\\nonumber&\cong\dots\\\nonumber
& \cong\RHom(L_1^i,L_{n_i}^i)\cong\K\oplus \K[-1].
\end{align}
Finally, if we apply $\RHom(L_1^i,-)$ to \eqref{eq:rlrlq}, for $j=1$, and we use \eqref{eqn:isograd} and the fact that
\[
\RHom(L_1^i,Q^i_{1,2}[-1])\cong\K[-2],
\]
then we get the isomorphisms of graded vector spaces
\[
\RHom(\mathsf{S}_\cL(L^i_1),L^i_1)\cong \RHom(L^i_1,\mathsf{S}_\cL^{-1}(L^i_1))\cong\RHom(L^i_1,\mathsf{R}_{\langle L_2^i,\dots,L_{n_i}^i\rangle}(L_1^i))\cong\K\oplus \K[-1]\oplus \K[-2].
\]
For the penultimate isomorphism we used \eqref{eqn:Serreright}.

Finally, let us prove (3). When $n_i=1$,  we have $\mathsf S^{-1}_\cL(L^i_1)=\mathsf S^{-1}_{\cL_i}(L^i_1)=L^i_1$ and the only isomorphim in the statement hold automatically.

When $n_i\geq 2$, since $Q^i_{j,j+1}\in \mathsf{S}_X(\cL)$, for every $1\leq j\leq n_i-1$, we have $\zeta^!_\cK(Q^i_{j,j+1})=0$. By definition, $\zeta^!_\cK(L_j^i)\cong\zeta^!_\cK(L_{j+1}^i)$, for every $1\leq j\leq n_i-1$. Therefore, if we apply the functor $\zeta^!_\cK$ to \eqref{eq:rlrlq}, for every $1\leq j\leq n_i-1$, we get $\zeta^!_\cK\circ\mathsf{S}^{-1}_\cL(L_1^i)\cong\zeta^!_\cK(L_{n_i}^i)$.
\end{proof}

\section{Classifying spherical and pseudoprojective objects}\label{sect:classification}

In this section we study and classify $3$-spherical and $3$-pseudoprojective objects in the Kuznetsov component of the special semiorthogonal decompositions in Section \ref{subsect:Enriques}.

\subsection{Spherical and pseudoprojective objects: construction and classification}\label{subsect:constructing}

We begin with a fairly general definition.

\begin{Def}\label{def:sphpseudosph}
Let $\cT$ a triangulated category that it is linear over a field $\K$ and with Serre functor $\mathsf{S}_\cT$.
\begin{enumerate}

\item[(a)] An object $E$ in $\cT$ is \emph{$n$-spherical} if:
\begin{enumerate}
\item[(i)] There is an isomorphism of graded vector spaces $\RHom(E,E)\cong\K \oplus \K[-n]$;
\item[(ii)] $\mathsf{S}_\cT(E)\cong E[n]$.
\end{enumerate}
\item[(b)]An object $E$ in $\cT$ is \emph{$n$-pseudoprojective} if:
\begin{enumerate}
\item[(i)] There is an isomorphism of graded vector spaces $\RHom(E,E)=\K\oplus \K[-1]\oplus\dots\oplus \K[-n]$.;
\item[(ii)] $\mathsf{S}_\cT(E)\cong E[n]$.
\end{enumerate}
\end{enumerate}
\end{Def}

\begin{Rem}\label{rmk:names}
Some comments on the choice of the names in the above definition are in order here. Spherical objects were introduced in the seminal paper \cite{ST} and our definition is exactly the same. The name is clearly motivated by the fact the the $\Ext$-algebra of such an object is isomorphic to the cohomology algebra of an $n$-sphere.

The situation is slightly different for $n$-pseudoprojective objects as in (b) of the definition above. Our objects look very much like a variant of $\mathbb{P}^n$-objects described in \cite{HT}. The first key different is that in \cite{HT} the $\Ext$-algebra of a $\mathbb{P}^n$-object is assumed to be isomorphic to the cohomology algebra of the complex projective  space $\mathbb{P}^n$ and it is then generated in degree $2$. On the contrary, our $n$-pseudoprojective objects have non-trivial $\Ext$'s in odd degrees as well. The case of objects with $\Ext$-algebra isomorphic to the cohomology algebra of $\mathbb{P}^n$ but with generator in degree $1$ has been studied in \cite{Krug}. Indeed, our objects $E$ in Definition \ref{def:sphpseudosph} (b) look very much like $\mathbb{P}^n[1]$-objects in \cite{Krug}. The key difference is that $\RHom(E,E)\cong\C[x]/(x^{n+1})$, with $x$ in degree $1$, only as graded vector spaces and not as algebras. Of course, we expect the objects studied in Theorem \ref{thm:sphericals} to be $\mathbb{P}^3[1]$-objects but, at the moment, we have no control on their $\Ext$-algebra and, as we will see, this is not relevant for our computations. For these reasons we prefer to adopt a new name and call them $3$-pseudoprojective with the hope that future studies may turn them into $\mathbb{P}^3[1]$-objects.
\end{Rem}

In this paper we will be mainly interested in $3$-spherical and $3$-pseudoprojective objects. The former were extensively used in \cite{LNSZ} to prove the Refined Derived Torelli Theorem in the generic case.

In the more general setting of the present paper, we want to show how one can possibly construct $3$-pseudoprojective objects in the Kuznetsov component of an Enriques surface whose derived category is endowed with a semiorthogonal decomposition as in \eqref{eqn:semicond}. To this extent, we set
\[
S_i:=\zeta^!_\cK(L^i_1).
\]

\begin{Rem}\label{rmk:Sifirst}
(i) By \eqref{eqn:defofleftmut}, the object $S_i$ sits in the distinguished triangle
\[
\zeta^!_\cK (L_1^i)\rightarrow L^i_1\to\mathsf{L}_{\Ku(X,\cL)}(L^i_1).
\]
By Lemma \ref{lem:distriangle}, we get the distinguished triangle
\begin{equation}\label{eqn:leftzeta}
\zeta^!_\cK (L_1^i)\rightarrow L^i_1\xrightarrow{\varphi_i}\mathsf{S}_X(\mathsf{S}_\cL^{-1}(L_1^i)).
\end{equation}
In other words, since $\Hom(L_1^i,\mathsf{S}_X(\mathsf{S_{\cL}^{-1}}(L^i_1)))\cong\K$, the object $S_i[1]$ is isomorphic to the cone of the unique non-trivial morphism $L_1^i\to\mathsf{S}_X(\mathsf{S_{\cL}^{-1}}(L^i_1))$.

(ii) Note that if $n_i\geq 2$, then by Lemma \ref{lem:compute} (3), we get an isomorphism $S_i\cong \zeta^!_{\cK}(L_j^i)$, for all $1\leq j\leq n_i$.
\end{Rem}

We can now discuss some additional properties of the objects $S_i$'s.

\begin{Lem}\label{lem:compute2}
In the above setting we have:
\begin{enumerate}
    \item [\rm{(1)}] If $n_i=1$, then $S_i$ is a $3$-spherical object in $\Ku(X,\cL)$.

\item [\rm{(2)}] If $n_i\geq 2$, then $S_i$ is a $3$-pseudoprojective object in $\Ku(X,\cL)$.
\end{enumerate}
Furthermore, if $i\neq j$, then $\RHom(S_i,S_j)=0$ and thus $S_i\not\cong S_j[k]$, for all $k\in\Z$.
\end{Lem}

\begin{proof}
If $n_i=1$, then the result is simply \cite[Lemma 4.8]{LNSZ}. Thus we can assume $n_i\geq 2$. In this case, we can first compute the action of the Serre functor using again an argument very close to the one in \cite[Lemma 4.8]{LNSZ} and which consists of the following chain of isomorphisms
\begin{align*}
    \mathsf{S}_{\cK} (S_i) & \cong \zeta^!_{\cK}(\mathsf{S}_X(S_i))\\
    &\cong \zeta^!_{\cK}(\Cone(\mathsf{S}_X(L_1^i)\xrightarrow{\mathsf{S}_X(\varphi_i)}\mathsf{S}_X^2(\mathsf{S}^{-1}_\cL( L_1^i)))[-1]) \\ 
    & \cong \zeta^!_{\cK}(\Cone(\mathsf{S}_X(L^i_1)\xrightarrow{\mathsf{S}_X(\varphi_i)}\mathsf S^{-1}_\cL (L^i_1)[3]))  \\
    & \cong  \zeta^!_{\cK} (\mathsf S^{-1}_\cL (L_1^i[3]))\\&\cong S_i[3],
\end{align*}
where the first isomorphism is \cite[Lemma 2.6]{KuzCYcat}, the second one is Remark \ref{rmk:Sifirst} (i), the third one follows from the the fact that $\mathsf{S}_X^2\cong[4]$ and the fourth one follows from the observation that $\zeta^!_\cK(\mathsf{S}_X(L_1^i))=0$ (this by \eqref{eqn:defofleftmut} and noticing that $\mathsf{S}_X(L_1^i)\in\cK^\perp$). Finally, the last isomorphism is a consequence of Lemma \ref{lem:compute} (3). This yields property (ii) in Definition \ref{def:sphpseudosph} (b).

To conclude that $S_i$ is $3$-pseudoprojective, we first apply $\Hom(S_i,-)$ to \eqref{eqn:leftzeta}. Hence, since $S_i\in \cK= \!^\perp (\mathsf S_X(\cL))$, we have an isomorphism of graded vector spaces
\begin{equation}\label{eqn:isograd2}
\RHom(S_i,S_i)\cong\RHom(S_i,L^i_1).
\end{equation}
By Lemma \ref{lem:compute} and Serre duality,
\[
\RHom(\mathsf{S}_X\circ\mathsf{S}^{-1}_\cL (L^i_1),L^i_1)\cong(\RHom(L^i_1,\mathsf S^{-1}_\cL (L^i_1)[4]))^\vee\cong\K[-2]\oplus \K[-3]\oplus \K[-4]. 
\]
Finally, if we apply the functor $\RHom(-,L^i_1)$ to the distinguished triangle \eqref{eqn:leftzeta}, then we get
\[
\RHom(S_i,L^i_1)\cong\K\oplus \K[-1]\oplus \K[-2]\oplus \K[-3].
\]
This, together with \eqref{eqn:isograd2}, implies (i) in Definition \ref{def:sphpseudosph} (b).
 
To prove the last claim in the statement, observe that, as in the previous part, we have the isomorphism of graded vector spaces $\RHom(S_i,S_j)\cong\RHom(S_i,L^j_1)$. If $i\neq j$, then $\RHom(L_1^i,L^j_1)=0$. On the other hand, by Serre duality, 
\[
\RHom(\mathsf{S}_X(\mathsf{S}^{-1}_\cL (L_1^i)),L^j_1)\cong\RHom(L^j_1,\mathsf{S}^{-1}_\cL (L_1^i)[4])^\vee=0,
\]
where the last vanishing is due to the fact that $\mathsf{S}^{-1}_\cL (L_1^i)$ is in $\cL_i$ which is orthogonal to $\cL_j$ to which $L^j_1$ belongs. By \eqref{eqn:leftzeta}, these two observations imply $\RHom(S_i,S_j)=0$, when $i\neq j$, and thus the claim follows.
\end{proof}

As we learnt from the previous lemma, we need to distinguish when $n_i=1$ and $n_i\geq 2$ in the semiorthogonal decomposition \eqref{eqn:semicond}. Thus it is convenient to reformulate here and keep it in mind for the rest of the paper the setup we are working in:

\begin{Set}\label{setup}
Let $X$ be an Enriques surface with a semiorthogonal decomposition
\[
\Db(X)=\langle\Ku(X,\cL),\cL\rangle,
\]
where $\cL=\langle\cL_1,\dots,\cL_c\rangle$ is an admissible subcategory and $\cL_1,\dots,\cL_c$ are orthogonal admissible subcategories and
\[
\cL_i=\langle L^i_1,\dots,L^i_{n_i}\rangle
\]
where
\begin{itemize}
\item[{\rm (a)}] $L_j^i$ is a line bundle such that $L^i_j=L^i_1\otimes\OO_X(R^i_1+\dots+R^i_{j-1})$ where $R^i_1,\dots,R^i_{j-1}$ is a chain of $(-2)$-curves of $A_{j-1}$ type; 
\item[{\rm (b)}] $\{L^i_1,\dots,L^i_{n_i}\}$ is an exceptional collection; and
\item[{\rm (c)}] If there is $j\in\{1,\dots,c\}$ such that $n_j\neq 1$, then there is $d\in\{1,\dots,c\}$ such that $n_j\geq 2$ for all $1\leq j\leq d$ while $n_j=1$ for all $d<j\leq c$.
\end{itemize}
\end{Set}

\begin{Rem}\label{rmk:noimp}
Since the components $\cL_i$'s are orthogonal, the special choice in (c) above for the ordering of these components does not cause loss of generality.
\end{Rem}

The goal of the rest of this section is to prove the following result which is the key for our Refined Derived Torelli Theorem and answers the open question in \cite[Remark 4.11]{LNSZ}.

\begin{Thm}\label{thm:sphericals}
In Setup \ref{setup}, if $F$ is an object in $\Ku(X,\cL)$, then
\begin{enumerate}
    \item[{\rm (1)}] $F$ is $3$-spherical if and only $F\cong S_j[k]$ for some $d<j\leq c$ and $k\in \Z$; \item[{\rm (2)}] $F$ is $3$-pseudoprojective if and only if $F\cong S_j[k]$ for some $1\leq j\leq d$ and $k\in \Z$.
\end{enumerate}
Furthermore, all these $3$-spherical and $3$-pseudoprojective objects are not isomorphic.
\end{Thm}

Before moving to the proof of this result, we want to discuss an easy application that shows that the Kuznetsov component of an Enriques surface is not, in general, intrinsic to the surface, up to equivalence.

\begin{Cor}\label{cor:nonintri}
If $X$ is a generic nodal Enriques surface, then there exist two semiorthogonal decompositions
\[
\Db(X)=\langle\Ku(X,\cL_i),\cL_i\rangle,
\]
for $i=1,2$, as in Setup \ref{setup} and such that $\cL_1$ and $\cL_2$ consist of $10$ exceptional line bundles but  $\Ku(X,\cL_1)\not\cong\Ku(X,\cL_2)$.
\end{Cor}

\begin{proof}
Consider the two distinct semiorthogonal decompositions on $\Db(X)$, for $X$ a generic nodal Enriques surface, described in Example \ref{ex:twoFano}. Assume that there is an exact equivalence $\Ku(X,\cL_1)\cong\Ku(X,\cL_2)$. Then the two Kuznetsov components would contain, up to shifts and isomorphisms, the same number of $3$-spherical and $3$-pseudoprojective objects. But this contradicts Theorem \ref{thm:sphericals}.
\end{proof}

\begin{Rem}\label{rmk:JH}
If one could prove that, given a generic nodal Enriques surface, the Kuznetsov components of the two semiorthogonal decompositions in Example \ref{ex:twoFano} have no non-trivial semiorthogonal decompositions, then Corollary \ref{cor:nonintri} would yield another counterexample to the \emph{Jordan\textendash H\"older property} of semiorthogonal decompositions. Roughly, such a property predicts that if $X$ is a smooth projective variety then the semiorthogonal decompositions of $\Db(X)$ are essentially unique, up to reordering of the components and equivalence. It is worth recalling that we already know counterexamples to such a property \cite{BBS,KuzJH}.
\end{Rem}

\subsection{Proof of Theorem \ref{thm:sphericals}}\label{subsect:proofsph}

It is clear that the `if' part in (1) and (2) and the last part of the statement in Theorem \ref{thm:sphericals} are the content of Lemma \ref{lem:compute2}. Thus it remains to prove the hard implication in (1) and (2): if $F$ is either a $3$-spherical or $3$-pseudoprojective object in $\Ku(X,\cL)$, then, up to shift, it is isomorphic to one of the $S_i$'s. The proof is split in several steps.

\subsubsection*{Step 1: reduction to one component}

We first prove that we can simplify our computation and reduce to the case $\cL=\cL_i$, for some $i=1,\dots,c$. For now, we can indifferently assume that $F$ is either a $3$-spherical or a $3$-pseudoprojective object in $\Ku(X,\cL)$.

First of all, note that the semiorthogonal decomposition in Setup \ref{setup} can be conveniently rewritten as
\[
\Db(X)=\langle \mathsf{S}_X(\cL),\Ku(X,\cL)\rangle.
\]
By \eqref{eqn:defofleftmut} and Lemma \ref{lem:distriangle}, we have a distinguished triangle
\begin{equation}\label{eq:Fseq}
    \zeta^!_{\mathsf{S}_X(\cL)}(F)\rightarrow F \rightarrow \mathsf{S}_X(\mathsf{S}^{-1}_{\Ku(X,\cL)}(F)).
\end{equation}
If $\zeta_{\mathsf{S}_X(\cL)}\colon\mathsf{S}_X(\cL)\hookrightarrow\Db(X)$ denotes the embedding of the corresponding admissible subcategory, then we set
\[
G:=\mathsf{S}_X(\zeta^!_{\mathsf{S}_X(\cL)}(F))\in\cL.
\]
We now want to prove some relevant properties of $G$.

Note that, by assumption, $\mathsf{S}_{\Ku(X,\cL)}(F)\cong F[3]$. Thus, by \cite[Lemma 2.6]{KuzCYcat}, we have the isomorphisms
\begin{align}\label{eqn:assum1}
    \mathsf{S}_{\cL}( G) &\cong \zeta^!_{\cL}(\mathsf S_X(G))\\\nonumber
    &\cong\zeta^!_{\cL}(\Cone(\mathsf{S}^2_X(F)\to\mathsf{S}^3_X(\mathsf{S}^{-1}_\cK( F)))[-1] )\\\nonumber
    & \cong \zeta^!_{\cL}(\Cone(F[3]\to\mathsf{S}_X(F))  )\\\nonumber
    & \cong  \zeta^!_{\cL}( \mathsf{S}_X(F))\\\nonumber
    & \cong \mathsf{S}_X(\zeta^!_{\mathsf{S}_X(\cL)}(F))\\\nonumber
    &\cong G,
\end{align}
where for the second one we use \eqref{eq:Fseq}. The third isomorphism follows from the fact that $\mathsf{S}_X^2(F)\cong F[4]$ and $\mathsf{S}_{\Ku(X,\cL)}(F)\cong F[3]$. The fourth is a consequence of $\zeta^!_{\cL}(F)=0$ while the penultimate is a simple computation. Here $\zeta_\cL\colon\cL\hookrightarrow\Db(X)$ is the embedding of the admissible subcategory $\cL$.

Furthermore, if we apply the equivalence $\mathsf{S}_X$ to \eqref{eq:Fseq}, we get the distinguished triangle
\begin{equation}\label{eq250}
    G\rightarrow \mathsf S_X(F)\rightarrow F[1].
\end{equation}
This implies that $[G]=2[F]$ in $\Num(X)$ and thus $\chi(G,G)=4\chi(F,F)=0$ because $F$ is either $3$-spherical or $3$-pseudoprojective. By \eqref{eqn:assum1}, we have $\Hom(G,G[t])\cong\Hom(G,G[-t])$ for every $t\in\Z$ and thus $\Hom(G,G)$ has dimension at least $2$.

If we apply $\RHom(G,-)$ to \eqref{eq250} and we use Serre duality, we get the isomorphism of graded vector spaces $\RHom(G,G)\cong\RHom(G,\mathsf{S}_X(F))$. Furthermore, if we apply $\RHom(-,\mathsf{S}_X(F))$ to \eqref{eq250} and we take into consideration that, by assumption,
\[
\RHom(F,\mathsf S_X (F))\cong\begin{cases}\K\oplus\K[3] & \text{ if } F \text{ is $3$-spherical;}\\  \K\oplus\K[1]\oplus\K[2]\oplus \K[3] & \text{ if } F \text{ is $3$-pseudoprojective,}\end{cases}
\]
then we get the isomorphisms of graded vector spaces
\begin{equation}\label{eqn:assum2}
\RHom(G,G)\cong\begin{cases}\K[-3]\oplus \K^{\oplus 2}\oplus\K[3] & \text{ if } F \text{ is $3$-spherical;}\\\K[\pm 3]\oplus \K[\pm 2]\oplus\K[\pm 1]\oplus \K^{\oplus 2} & \text{ if } F \text{ is $3$-pseudoprojective;}\end{cases}
\end{equation}
Here we used that $\Hom(G,G)$ has dimension at least $2$.

Denote by $\cQ$ the collection of torsion sheaves $Q^i_{j,l}$ for all $1\leq i\leq c$ and $1\leq j<l\leq n_i$, defined in \eqref{eqn:Qij}. Since  both $F$ and $\mathsf S_X( F)$ are in $\cQ^\perp={}^\perp\cQ$ (here we use Remark \ref{rmk:Qs}), by \eqref{eq250} both $G$ and $\mathsf{S}_X(G)$ are orthogonal to $\cQ$ as well. Thus
\begin{equation}\label{eqn:assum3}
\RHom( G,Q)\cong\RHom(\mathsf S_\cL(G),Q)=0,
\end{equation}
for every $Q\in\cQ$.

We now prove the following result.

\begin{Lem}\label{lem:homsofG}
In Setup \ref{setup}, let $E$ be an object in $\cL$ satisfying
\begin{enumerate}
    \item[{\rm (a)}] $\mathsf{S}_\cL(E)\cong E$;
    \item[{\rm (b)}] We have the isomorphisms $\Hom(E,E)\cong\K^{\oplus 2}$; $\Hom(E,E[3])\cong\K$; $\Hom(E,E[1])\cong\Hom(E,E[2])$ and $\Hom(E,E[t])=0$, for $\abs{t}\geq 4$;
    \item[{\rm (c)}] $\Hom(E,Q)=0$, for every $Q\in \cQ$.
\end{enumerate}
Then 
\begin{enumerate}
    \item [\rm{(1)}] The object $E\in\cL_i$ for some $i$;
    \item [\rm{(2)}] There exists $t\in \Z$ such that for all $1\leq j\leq n_i$, $\Hom(E,L^i_j[t])\cong\Hom(E,L^i_j[t+3])\cong\K$ and $\Hom(E,L^i_j[s])=0$ for $s\leq t-1$ and $s\geq t+4$;
    \item[\rm{(3)}] The composition of non-trivial morphisms $E\to E[3]$ and $E[3]\to L^i_1[t+3]$ is non-trivial.
\end{enumerate}
\end{Lem}

\begin{proof} 
Let us first prove (1). Since the $\cL_i$'s are orthogonal, we may write $E=F_1\oplus \dots\oplus F_c$, for $F_i$ in $\cL_i$. For the same reason,
\[
\sum \chi(F_i,F_i)=\chi(E,E)=0.
\]
Hence, by Lemma \ref{lem:compute} (1), $\chi(F_i,F_i)=0$ for every $1\leq i\leq c$.

Since $\mathsf{S}_\cL (E)\cong E$ and the $\cL_i$'s are orthogonal, we must have $\mathsf{S}_{\cL_i}(F_i)\cong F_i$, for every $i$. Therefore, $\Hom(F_i,F_i[t])\cong\Hom(F_i,F_i[-t])$, for every $t\in\Z$, and since $\chi(F_i,F_i)=0$, we must have that $\Hom(F_i,F_i)$ is even-dimensional.
As $\Hom(E,E)=\bigoplus_{1\leq i\leq d}\Hom(F_i,F_i)$, assumption (b) implies that $E\cong F_i$, for a unique $i$. Thus $E$ is in $\cL_i$ and we get (1).

In view of what we have just proved, we can simplify the notation and assume $\cL=\cL_i$ and set $n:=n_i$, $L_j:=L^i_j$ and $Q_{j,l}:=Q^i_{j,l}$. First, we can note that, if $n=1$, then $E=L[t]\oplus L[t+3]$, for some $t\in\Z$. Thus (2) and (3) hold true automatically and we may assume $n\geq 2$ from now on.

Let us prove (2) under this assumption. We set $E_1:=E\in\cL$ and write it as an extension
\[
E_2\to E_1\to L_1\otimes\RHom(E_1,L_1)^\vee,
\]
where now $E_2\in\langle L_2,\dots,L_n\rangle$. Inductively, for all $j=1,\dots,n-1$, we then define
\begin{equation}\label{eqn:new1}
E_{j+1}\to E_j\to L_j\otimes\RHom(E_j,L_j)^\vee,
\end{equation}
where $E_k\in\langle L_k,\dots, L_n\rangle$.

If we apply the functor $\RHom(-,L_{j+1})$ to \eqref{eqn:new1} we get the exact sequence
\begin{equation}\label{eqn:new2}
\RHom(L_j,L_{j+1})\otimes\RHom(E_j,L_j)\to\RHom(E_j,L_{j+1})\to\RHom(E_{j+1},L_{j+1}).
\end{equation}
Now consider the distinguished triangle
\begin{equation}\label{eqn:triadop}
L_k\to L_{k+1}\to Q_{k,k+1},
\end{equation}
where the first map is the unique (up to scalar) non-trivial one, and apply $\RHom(E_j,-)$ to it. Note that, by construction, the object $E_j$ is obtained with a finite number of extensions involving $E$ and elements in $\langle L_1,\dots,L_{j-1}\rangle$. On the other hand, $\RHom(L_s,Q_{j,j+1})=0$, when $s\leq j-1$, while, by assumption (c), we have $\RHom(E,Q_{j,j+1})=0$. Therefore we get an isomorphism of graded vector spaces
\[
\Hom(L_j,L_{j+1})\otimes\RHom(E_j,L_j)\stackrel{\sim}{\longrightarrow}\RHom(E_j,L_{j+1}).
\]
Recall that $\Hom(L_j,L_{j+1})\cong\K$ with generator the non-trivial morphism in \eqref{eqn:triadop}.

If we plug this into \eqref{eqn:new2}, we get the isomorphism
\[
\RHom(E_{j+1},L_{j+1})\cong\RHom(E_{j},L_{j})
\]
and, iterating the argument, we get
\[
\RHom(E_{j},L_{j})\cong\RHom(E,L_{1}),
\]
for all $j=1,\dots,n-1$. In particular, \eqref{eqn:new1} gets the form
\begin{equation}\label{eqn:new3}
E_{j+1}\to E_j\to L_j\otimes\RHom(E,L_1)^\vee.
\end{equation}

Again by assumption (c), Serre duality and the fact that $\mathsf{S}_\cL(E)=E$, we get
\begin{equation}\label{eq:2011}
\RHom(E,L_j)\cong\RHom(E,L_1)\quad\text{ and }\quad\RHom(L_j,E)\cong\RHom(L_1,E),
\end{equation} for every $j=1,2,\dots,n-1$. Hence, if we apply the functor $\RHom(E,-)$ to \eqref{eqn:new3} and we use this observation, then we get the distinguished triangle
\begin{equation}\label{eqn:new4}
\RHom(E,E_{j+1})\to\RHom(E,E_j)\to\RHom(E,L_1)\otimes\RHom(E,L_1)^\vee.
\end{equation}

Ut to shifting $E$, we may assume that there is a positive integer $m$ such that $\Hom(E,L_1[t])$ is non-trivial for $t=0,m$, and it is trivial for $t<0$ and $t>m$. Thus the graded vector space $\RHom(E,L_1)\otimes\RHom(E,L_1)^\vee$ has non-trivial components concentrated in degrees $-m,\dots,m$. Moreover, in degrees $-m$ and $m$, the corresponding components have the form
\[
\Hom(E,L_1)\otimes\Hom(E,L_1[m])^\vee\cong\Hom(E,L_1[m])\otimes\Hom(E,L_1)^\vee\cong\K^{\oplus a_0a_m},
\]
where $a_0:=\dim\Hom(E,L_1)$ and $a_m:=\dim\Hom(E,L_1[m])$.

On the one hand, we have
\[
\RHom(E,E_n)\cong\RHom(E,L_1)\otimes\RHom(E,L_1)^\vee.
\]
On the other hand, this information plugged into \eqref{eqn:new4} for $j=n-1$ yields that $\RHom(E,E_{n-1})$ has non-trivial component concentrated in degree $-m,\dots,m$ as well. Furthermore,
\[
\dim\Hom(E,E_{n-1}[\pm m])\geq a_0a_m.
\]
By descending induction on $j=n,\dots,1$ we get
\[
\dim\Hom(E,E_{j}[\pm m])\geq a_0a_m
\]
and thus the same holds true for the case $j=1$.

By assumption (b), we have $\Hom(E,E[3])\cong\K$ while $\Hom(E,E[t])=0$, for $t>3$. Thus $m=3$ and $a_0=a_m=1$ as claimed in (2).

We are finally ready to prove (3). As in the proof of (2), we can assume $t=0$ without loss of generality. If we apply $\RHom(E,-)$ to \eqref{eqn:new3} with $j=1$ (recall that $E_0=E$), we get the long exact sequence
\begin{equation}\label{eqn:long1}
\dots\rightarrow \Hom(E,E[3])\to\Hom(E,L_1[3])\to\Hom(E,E_1[4])\to 0
\end{equation}
where the first map is given by the composition with the unique (up to scalar) non-trivial morphism $E[3]\to L_1[3]$ from (2). Since $E_1$ is the extension of $L_a,\dots,L_a[3]$, for $2\leq a\leq n$, by \eqref{eq:2011} and the fact that $\RHom(E,L_1)$ has non-trivial components concentrated in degrees $0,\dots,m$, we get $\Hom(E,E_1[4])=0$. Hence, since $\Hom(E,E[3])\cong\Hom(E,L_1[3])\cong\K$, the composition with any non-trivial map $E[3]\rightarrow L_1[3]$, which defines the first morphism in \eqref{eqn:long1}, induces an isomorphism $\Hom(E,E[3])\cong\Hom(E,L_1[3])$. This is precisely (3).
\end{proof}

By \eqref{eqn:assum1}, \eqref{eqn:assum2} and \eqref{eqn:assum3}, the assumptions (a)--(c) of Lemma \ref{lem:homsofG} for the object $G$. Thus $G$ is in $\cL_i$, for some $1\leq i\leq c$. The other parts of Lemma \ref{lem:homsofG} will be use later in the proof.

\subsubsection*{Step 2: reduction to the pseudoprojective case ($i\leq d$)}
Let $\cL_i$ be the admissible subcategory identified in Step 1 and set $\hat{\cL}_i:=\langle \cL_1,\dots,\cL_{i-1},\cL_{i+1},\dots,\cL_d\rangle$. If we apply the functor $\mathsf{S}_X^{-1}$ to \eqref{eq:Fseq},  we get the isomorphisms $\mathsf{S}_X^{-1}(G)\cong\zeta^!_{\cL_i}(\mathsf{S}_X^{-1}(F))$ and $\mathsf S^{-1}_{\Ku(X,\cL)} (F)\cong\zeta^*_{\langle\Ku(X,\cL),\hat{\cL}_i\rangle}(\mathsf S^{-1}(F))$. By \cite[Lemma 2.6]{KuzCYcat}, this yields
\begin{align*}
    \mathsf{S}^{-1}_{\langle\Ku(X,\cL),\hat{\cL}_i\rangle}(F)\cong\zeta^*_{\langle\Ku(X,\cL),\hat{\cL}_i\rangle}(\mathsf{S}_X^{-1}(F))\cong\mathsf{S}^{-1}_{\Ku(X,\cL)}( F)\cong F[-3].
\end{align*}
Here $\zeta_{\cL_i}$ and $\zeta_{\langle\Ku(X,\cL),\hat{\cL}_i\rangle}$ are the embeddings of the corresponding admissible subcategories.
Since condition (i) in Definition \ref{def:sphpseudosph} is clearly satisfied, this implies that $F$ is $3$-spherical or $3$-pseudoprojective in the larger category $\langle\Ku(X,\cL),\hat{\cL}_i\rangle$ as well. 

Assume now $i>d$. By Lemma \ref{lem:homsofG} (2) applied to $G=\mathsf{S}_X(\zeta^!_{\mathsf{S}_X(\cL)}(F))\in\cL_i$, we must have $G\cong L_i[t]\oplus L_i[t+3]$ for some $t\in \Z$. In particular, by \eqref{eqn:assum2}, the object $F$ is $3$-spherical. By \cite[Proposition 4.10]{LNSZ}, $F\cong S_i[k]$ for some $k\in \Z$. Therefore, from now on, we may assume $i\leq d$.

\subsubsection*{Step 3: the final isomorphism}
First note that, for $i$ as in Step 2 and since $\mathsf{S}_X(F)\in{}^\perp\!\cL$, by Lemma \ref{lem:homsofG} and \eqref{eq250}, we have the isomorphisms
\begin{equation*}
\Hom(F,L^i_1[t])\cong \Hom(G,L^i_1[t])\cong\begin{cases}\K& \text{if $t=0,3$,}\\ 0 & \text{if $t\leq -1$ and $t\geq 4$,}\end{cases}
\end{equation*}
up to shifting $F$ (and thus $G$).

If $\psi\colon F\to L_1^i$ is the unique, up to scalars, non-trivial morphism, we set
\[
C:=\mathrm{Cone}(F\xrightarrow{\psi} L^i_1).
\]
and get the distinguished triangle
\begin{equation}\label{eqn:disttrian}
F\xrightarrow{\psi} L_1^i\xrightarrow{\varphi}C.
\end{equation}
We first prove the following result.

\begin{Lem}\label{lem:diamond}
In the notation above, $C\cong \mathsf{S}_X(\mathsf{S_{\cL}^{-1}}(L^i_1))$.
\end{Lem}

\begin{proof}
First, we apply $\RHom(-,L^i_1)$ to \eqref{eqn:disttrian}. By definition of $\psi$, $\Hom(L^i_1,L^i_1)\xrightarrow{\psi}\Hom(F,L^i_1)$ is an isomorphism.  Since $\Hom(F[1],L^i_1)=0$ and $\Hom(L^i_1,L^i_1[1])=0$, we get $\Hom(C,L^i_1)=0$ and $\Hom(C,L^i_1[1])=0$.

If we apply $\RHom(C,-)$ to \eqref{eqn:disttrian}, we get the isomorphisms
$\Hom(C,C)\cong\Hom(C,F[1])\cong\K$. Since $\mathsf{S}_X(\mathsf{S}_{\cL}^{-1}(L^i_1))$ is exceptional, being $L^i_1$ so, the claim follows if we show that there are non-trivial morphisms
\[
\mathsf{S}_X(\mathsf{S}_{\cL}^{-1}(L^i_1))\stackrel{g}{\longrightarrow} C\stackrel{f}{\longrightarrow} \mathsf{S}_X(\mathsf{S}_{\cL}^{-1}(L^i_1))
\]
whose composition is non-trivial.

To proceed in this direction, we first observe that, if we apply $\RHom(-,F[1])$ to \eqref{eqn:disttrian} and we use Serre duality, we get
\begin{align*}\label{eqn:Hom1}
&\Hom(C,F[1])\cong\Hom(F[1],F[1])\cong\K\\
&\Hom(\mathsf{S}_X(F),C)\cong\Hom(C,F[4])^\vee\cong\Hom(F[1],F[4])^\vee\cong\K.
\end{align*}
Indeed, the fact that $F\in\cL^\perp$ yields $\RHom(C,F)\cong\RHom(F[1],F)$.

Let $k_1$ be a non-trivial element in $\Hom(C,F[1])$. Note that since $L^i_1$ is exceptional, up to scalars, the morphism $k_1$  realizes $L^i_1$ as the extension of $F$ and $C$ in \eqref{eqn:disttrian} yielding the distinguished triangle
\begin{equation}\label{eqn:disttrian1}
L_1^i\xrightarrow{\varphi}C\xrightarrow{k_1} F[1]\xrightarrow{\psi[1]}L_1^i[1].
\end{equation}
Let $k_2$ be the non-trivial generator of $\Hom(\mathsf{S}_X(F),C)$.

On the other hand, as $\mathsf{S}_X(F)\in \langle\mathsf{S}_X(\mathsf{S_{\cL}^{-1}}(L^i_1))\rangle^\perp$, by \eqref{eq250} we get 
\begin{align*}\label{eqn:Hom2}
\Hom(\mathsf{S}_X(\mathsf{S}_{\cL}^{-1}(L^i_1)),F[1])&\cong\Hom(\mathsf{S}_X(\mathsf{S_{\cL}^{-1}}(L^i_1)),G[1])\\\nonumber
&\cong\Hom(G[1],\mathsf{S}_{\cL}^{-1}(L^i_1)[4])^\vee\\\nonumber
&\cong\Hom(G,L^i_1[3])^\vee\\\nonumber
&\cong\K,
\end{align*}
where the second isomorphisms is by Serre duality while the penultimate is obtained by applying $\mathsf{S}_\cL$ and using that $\mathsf S_\cL (G)=G$ by \eqref{eqn:assum1}.
Let $k_3$ be a non-trivial generator of the vector space $\Hom(\mathsf{S}_X(\mathsf{S}_{\cL}^{-1}(L^i_1)),F[1])$.

Finally, we consider the isomorphisms
\[
\Hom(\mathsf{S}_X(F),\mathsf{S}_X(\mathsf{S}_{\cL}^{-1}(L^i_1)))\cong\Hom(F,\mathsf{S}_{\cL}^{-1}(L^i_1))\cong\Hom(G,\mathsf{S}_{\cL}^{-1}(L^i_1))\cong\Hom(G,L^i_1)\cong\K,
\]
where the first isomorphism is due to the fact that $\mathsf{S}_X$ is an equivalence. For the second one, we apply $\RHom(-,\mathsf{S}_\cL (L^i_1))$ to \eqref{eq250} and use that $\mathsf{S}_X(F)\in{}^\perp\!\cL$. For the third one, we apply $\mathsf{S}_\cL$ and use \eqref{eqn:assum1}. The last isomorphism is due to Lemma \ref{lem:homsofG} (2). We set $k_4$ to be the non-trivial generator of $\Hom(\mathsf{S}_X(F),\mathsf{S}_X(\mathsf{S}_{\cL}^{-1}(L^i_1)))$.

Let us prove that $k_3\circ k_4$ is a non-trivial morphism. By Lemma \ref{lem:homsofG} (3), the composition of the unique (up to scalars) non-trivial morphisms
\[
G\to G[3]\to L^i_1[3]
\]
is non-trivial. Call it $k$. If we apply $\mathsf{S}^{-1}_\cL$ to the whole composition and since $G\cong\mathsf{S}^{-1}_\cL(G)$, then the composition of non-zero morphisms
\[
G\to G[3]\to\mathsf{S}^{-1}_\cL( L^i_1)[3]
\]
is non-trivial. Denote by $h$ the morphism from $F$ to $G$ in \eqref{eq250}. Since $\mathsf S(F)\in{}^\perp\!\cL$, the composition of $h$ with the non-trivial morphisms $G\to G[3]$ and $G\to\mathsf{S}^{-1}_\cL (L^i_1)[t]$ are  non-trivial, for $t=0,3$. Indeed, it is enough to apply $\RHom(-,G[3])$ and $\RHom(-,\mathsf{S}_\cL(L_1^i)[t])$ to \eqref{eq250}. Combining these two remarks we see that the composition of the non-trivial morphisms
\[
F\to G[3]\to\mathsf{S}_\cL^{-1}(L_1^i)[3]
\]
is non-trivial because it factors as
\[
F\xrightarrow{h}\underbrace{G\to G[3]\to\mathsf{S}_\cL^{-1}(L_1^i)[3]}_{k}.
\]
 
Since $\Hom(F,\mathsf{S}_X(F)[3])\cong\Hom(F,\mathsf{S}_X(F)[4])=0$, the composition of non-trivial morphisms
\[
F\to F[3]\xrightarrow{h[3]} G[3]
\]
is non-zero and it coincides with the non-trivial morphism $F\to G[3]$ above, up to a scalar. Therefore, the composition of non-trivial morphisms
\[
F\to F[3]\xrightarrow{h[3]} G[3]\to\mathsf S^{-1}_\cL (L^i_1)[3]
\]
is non-zero and the same is true for the composition of non-trivial morphisms
\begin{equation}\label{eqcomp}
F\to F[3]\to\mathsf S^{-1}_\cL( L^i_1)[3].
\end{equation}
Set
\[
K:=\Cone\left(F\xrightarrow{\mathsf{S}_X^{-1}(k_4)}\mathsf{S}^{-1}_\cL (L^i_1)\right).
\]
Since the composition in \eqref{eqcomp} is non-zero and $\Hom(F,F[4])=0$, we have $\Hom(F,K[3])=0$. By Serre duality, $\Hom(\mathsf{S}_X(K),F[1])=0$. Finally, if we apply $\RHom(-,F[1])$ to the distinguished triangle:
\[
\mathsf{S}_X(F)\stackrel{k_4}{\longrightarrow}\mathsf{S}_X(\mathsf{S}^{-1}_\cL(L^i_1))\longrightarrow \mathsf{S}_X(K),
\]
we get
\begin{equation}\label{eqn:neq}
0\neq k_3\circ k_4\colon\mathsf{S}_X(F)\to\mathsf{S}_X(\mathsf{S}^{-1}_\cL(L^i_1))\to F[1].
\end{equation}

By Serre duality and Lemma \ref{lem:compute}, $\Hom(\mathsf{S}_X(\mathsf{S}^{-1}_\cL (L^i_1)),L^i_1[t])\cong\Hom(\mathsf{S}_\cL(L^i_1)[t],L^i_1[4])=0$, when $t\leq 1$. Thus, if we apply $\RHom(\mathsf{S}_X(\mathsf S^{-1}_\cL( L^i_1)),-)$ to \eqref{eqn:disttrian1}, we get an isomorphism
\[
k_1\circ -: \Hom(\mathsf{S}_X(\mathsf{S}^{-1}_\cL( L^i_1)),C)\stackrel{\sim}{\longrightarrow}\Hom(\mathsf{S}_X(\mathsf S^{-1}_\cL (L^i_1)),F[1]).
\]
In particular, $ \Hom(\mathsf{S}_X(\mathsf{S}^{-1}_\cL (L^i_1)),C)\cong\K$ and there exists a non-trivial map $g\colon\mathsf{S}_X(\mathsf{S}^{-1}_\cL (L^i_1))\to C$ such that
\begin{equation}\label{eqn:neq1}
k_1\circ g=k_3\neq 0.
\end{equation}

Let us now produce the morphism $f$. Consider the diagram 
\begin{equation}\label{eqn:bigsquare}
	\begin{tikzcd}
		\mathsf{S}_X(F)[-1] \arrow{d} \arrow{r}
		& F \arrow{r} \arrow{d} & G\arrow{d}\arrow{r} & \mathsf{S}_X(F) \arrow{d}{}\\
		0 \arrow{d} \arrow{r}
		& L^i_1 \ar[equal]{r} \arrow{d}{} & L^i_1\arrow{d}\arrow{r}& 0 \arrow{d}{}\\
		\mathsf{S}_X(F) \arrow{r}{k_2}
		& C \arrow{r}  & M \arrow{r} & \mathsf{S}_X(F)[1],
	\end{tikzcd}
\end{equation}
where the first row is (a rotation of) the distinguished triangle \eqref{eq250}, the second column is \eqref{eqn:disttrian} while the third column is a distinguished triangle whose morphism $G\to L_1^i$ is the unique (up to scalars) non-trivial morphism from Lemma \ref{lem:homsofG} (2). Since $\mathsf{S}_X(F)\in {}^\perp\! \cL$, the top squares in the diagram commute and by the octahedron axiom, the bottom row is a distinguished triangle as well.

If we apply $\RHom(-,L^i_1)$ to the triangle in the third column of \eqref{eqn:bigsquare} and we take into account that $\Hom(G,L^i_1)\cong\K$ and $\Hom(G[1],L^i_1)\cong\Hom(L^i_1[-1],L^i_1)=0$ (by Lemma \ref{lem:homsofG} (2)), we get $\Hom(M,L^i_1)\cong\Hom(M[-1],L^i_1)=0$. By Serre duality (both in $\Db(X)$ and in $\cL$),
\[
\Hom(M[t],\mathsf{S}_X(\mathsf{S}^{-1}_\cL (L^i_1)))\cong\Hom(\mathsf{S}^{-1}_\cL (L^i_1),M[t])^\vee\cong\Hom(M[t],L^i_1)=0,
\]
when $t=0$ or $-1$.

Now, if we apply  $\RHom(-,\mathsf{S}_X(\mathsf{S}^{-1}_\cL(L^i_1)))$ to the distinguished triangle in the bottom row of \eqref{eqn:bigsquare} and we use the vanishing above, then we get an isomorphism
\[
-\circ k_2: \Hom(C,\mathsf{S}_X(\mathsf{S}^{-1}_\cL(L^i_1)))\stackrel{\sim}{\longrightarrow} \Hom(\mathsf{S}_X(F),\mathsf{S}_X(\mathsf{S}^{-1}_\cL(L^i_1))).
\]
In particular, $\Hom(C,\mathsf{S}_X(\mathsf{S}^{-1}_\cL(L^i_1)))\cong\K$ and there a non-trivial $f\colon C\to \mathsf{S}_X(\mathsf{S}^{-1}_\cL(L^i_1))$ such that
\begin{equation*}\label{eqn:neq2}
f\circ k_2=k_4 \neq 0.
\end{equation*}

To summarize, all morphisms we introduced so far fit in the following diagram
\[
\xymatrix{
&& \mathsf{S}_X(F)\ar[d]^{k_2}\ar[drr]^{k_4}\\
\mathsf{S}_X(\mathsf{S_{\cL}^{-1}}(L^i_1)) \ar[rr]^-{g}\ar[drr]_-{k_3}&& C \ar[rr]^-{f}\ar[d]^-{k_1}&&\mathsf{S}_X(\mathsf{S_{\cL}^{-1}}(L^i_1))\\
&&F[1],&&
}
\]
where the two triangles are commutative. Note now that the composition $k_3\circ f$ is non-trivial. Indeed, if not, we would have $0=k_3\circ f\circ k_2=k_3\circ k_4$, contradicting \eqref{eqn:neq}. Thus there is $0\neq\lambda\in\K$ such that $\overline{k}_3:=\lambda k_3\neq 0$ sits in the diagram
\[
\xymatrix{
\mathsf{S}_X(\mathsf{S_{\cL}^{-1}}(L^i_1)) \ar[rr]^-{g}\ar[drr]_-{k_3}&& C \ar[rr]^-{f}\ar[d]^-{k_1}&&\mathsf{S}_X(\mathsf{S_{\cL}^{-1}}(L^i_1))\ar[dll]^-{\overline{k}_3}\\
&&F[1],&&
}
\]
where the two triangles are commutative.

Assume now $0=f\circ g$. By the commutativity of the diagram above
\[
0=\overline{k}_3\circ f\circ g= k_1\circ g=k_3
\]
which contradicts \eqref{eqn:neq1}. Thus $f\circ g\neq 0$ as we want.
\end{proof}

In conclusion we get the following sequence of isomorphisms
\[
F[1]\cong\Cone(L^i_1\xrightarrow {\varphi}C)\cong \Cone(L^i_1\to\mathsf{S}_X(\mathsf{S_{\cL}^{-1}}(L^i_1)))\cong S_i[1].
\]
Note that for the first isomorphism follows from \eqref{eqn:disttrian}, the second one from Lemma \ref{lem:diamond} while the last one is a consequence of the definition of $S_i$ and of the isomorphism 
\[
\Hom(L^i_1,\mathsf{S}_X(\mathsf{S_{\cL}^{-1}}(L^i_1)) )\cong\K.
\]
This concludes the proof of Theorem \ref{thm:sphericals}.

\section{Proof of the main result}\label{sect:proof}

In this section, after a quick discussion about Fourier\textendash Mukai functors and the way we can extend them from admissible subcategories to larger subcategories, we prove our main result. As we will see the main theorem follows from the more precise statement in Section \ref{subsect:proofthm}.

\subsection{Fourier\textendash Mukai functors and their extensions}\label{subsect:criterion}
Let us start with a short introduction to the theory of Fourier\textendash Mukai functors. In complete generality, assume that $X_1$ and $X_2$ are smooth projective varieties over $\K$ with admissible subcategories 
\[
\alpha_i:\AA_i\into\Db(X_i),
\]
for $i=1,2$.

An exact functor $\mathsf{F}\colon\AA_1\to\AA_2$ is \emph{of Fourier\textendash Mukai type} if there exists an object $\mathcal{E}\in\Db(X_1\times X_2)$ such that there is an isomorphism of exact functors
\[
\alpha_2\circ \mathsf{F}\cong\Phi_\EE|_{\AA_1}\colon\AA_1\to\Db(X_2).
\]
The exact functor $\Phi_{\EE}$ is defined as
\[
\Phi_\mathcal{E}(-):=p_{2*}(\mathcal{E}\otimes p_1^*(-)),
\]
where $p_i\colon X_1\times X_2\to X_i$ is the $i$th natural projection.

\begin{Rem}\label{rmk:FM}
(i) Suppose we are given a Fourier\textendash Mukai functor  $\Phi_\cE\colon\Db(X_1)\to\Db(X_2)$ such that $\mathsf{F}:=\Phi_\cE|_{\AA_1}\colon\AA_1\to\Db(X_2)$ factors through $\AA_2$. The projection functor onto $\AA_i$ is of Fourier\textendash Mukai type   by \cite[Theorem 7.1]{Kuz11}. Thus, precomposing $\Phi_\cE$ with the projection onto $\AA_1$, yields a Fourier\textendash Mukai functor $\Phi_{\cE'}\colon\Db(X_1)\to\Db(X_2)$ such that $\Phi_{\cE'}|_{\AA_1}=\mathsf{F}$ and $\Phi_{\cE'}({}^\perp\!\AA_1)=0$.

(ii) It should be noted that, when $\mathsf{F}\colon\AA_1\to\AA_2$ is an equivalence, \cite[Conjecture 3.7]{Kuz07} should imply that $\mathsf{F}$ is of Fourier\textendash Mukai type in the above sense. This expectation is motivated by what is known for full functors between the bounded derived categories of smooth projective varieties (see \cite{COS,CS,Ol,Or}).
\end{Rem}

In general, one should not expect to be able to extend an equivalence between admissible subcategories to the whole triangulated categories. Nonetheless, this is possible under some compatibility assumptions as illustrated in the following result which we proved in \cite{LNSZ}.

\begin{Prop}[{\cite[Propositions 2.4 and 2.5]{LNSZ}}]\label{prop:extension}
Let $\alpha_1\colon\cA_1\hookrightarrow\Db(X_1)$ be an admissible embedding and let $E\in \!^\perp\!\cA_1$ with counit of adjunction  $\eta_1\colon\alpha_1\alpha_1^!(E)\to E$ .
Let $\Phi_\cE:\Db(X_1)\rightarrow \Db(X_2)$ be a Fourier\textendash Mukai functor with the property that $\Phi_\cE(^\perp\!\cA_1)\cong 0$.  Suppose further that
\begin{itemize}
\item[{\rm (a)}] $\Phi_\cE|_{\AA_1}$ is an equivalence onto an admissible subcategory $\AA_2$ with embedding $\alpha_2:\AA_2\into\Db(X_2)$, and
\item[{\rm (b)}] there is an exceptional object $F\in{}^\perp\!\AA_2$ and an isomorphism $\rho:\Phi_\cE(\alpha_1\alpha_1^!(E))\isomor\alpha_2\alpha^!_2(F)$.
\end{itemize}
Then there exists a Fourier\textendash Mukai functor $\Phi_{\tilde{\cE}}:\Db(X_1)\rightarrow \Db(X_2)$ satisfying
\begin{itemize}
    \item [\rm{(1)}] $\Phi_{\tilde{\cE}}(^\perp\langle \cA_1,E\rangle)\cong \mathsf 0$;
    \item [\rm{(2)}] $\Phi_{\tilde{\cE}}|_{\cA_1}\cong \Phi|_{\cA_1}$ and $\Phi_{\tilde{\cE}}(E)\cong F$;
    \item[\rm{(3)}] $\Phi_{\tilde{\cE}}|_{\langle \cA_1, E\rangle}$ is an equivalence onto $\langle \cA_2,F\rangle$.
\end{itemize}
\end{Prop}

\subsection{Proof of Theorem \ref{thm:derived_torreli}}\label{subsect:proofthm}

Let $X_1$ and $X_2$ be smooth Enriques surfaces over $\K$ admitting admissible subcategories $\cL$ and $\cM$ which fall under Setup \ref{setup}.  More precisely, their derived categories $\Db(X_1)$ and $\Db(X_2)$ admit semiorthogonal decompositions
\[
\Db(X_i)=\langle\Ku(X_1,\cL),\cL\rangle\qquad\text{ and }\qquad\Db(X_2)=\langle \Ku(X_2,\cM),\cM\rangle
\]
where  $\cL=\langle\cL_1,\dots,\cL_c\rangle$ and $\cM=\langle\cM_1,\dots,\cM_{c'}\rangle$ and 
\begin{align*}
    \cL_i=\langle L^i_1,\dots,L^i_{n_i}\rangle\qquad\text{ and }\qquad\cM_i=\langle M^i_1,\dots, M^i_{n'_i}\rangle
\end{align*}
are as that in Setup \ref{setup}.

We are now going to prove the following result which is actually a more precise version of Theorem \ref{thm:derived_torreli}.

\begin{Thm}\label{thm:gen}
Under the assumptions above, let $\mathsf{F}\colon \Ku(X_1,\cL)\to\Ku(X_2,\cM)$ be an equivalence which is of Fourier\textendash Mukai type. Then the two semiorthogonal decompositions have the same type and there exists a Fourier\textendash Mukai functor $\Phi_{\tilde{\cE}}\colon\Db(X_1)\to\Db(X_2)$ such that
\begin{itemize}
\item[{\rm (1)}] $\Phi_{\tilde{\cE}}|_{\Ku(X_1,\cL)}\cong\mathsf{F}$;
\item[{\rm (2)}] $\Phi_{\tilde{\cE}}\colon\Db(X_1)\to\Db(X_2)$ is an equivalence and thus an isomorphism $X_1\cong X_2$; and
\item[{\rm (3)}] Up to reordering, $\Phi_{\tilde{\cE}}(L^i_j)\cong M^i_j[t_i]$, for some $t_i\in \Z$, all $i=1,\dots,c$, and $j=1,\dots,n_i$. 
\end{itemize}
\end{Thm}

\begin{proof}
By Remark \ref{rmk:FM} (i), there exists $\cE\in \Db(X_1\times X_2)$ such that $\Phi_\cE|_{\Ku(X_1,\cL)}\cong \mathsf{F}$ and $\Phi_\cE(\cL)\cong \mathsf 0$. The proof proceeds now by induction on $c$, given that when $c=0$, there is nothing to prove.

First observe that, by Theorem \ref{thm:sphericals}, we get the identity $c=c'$ and, up to reordering, 
\begin{equation}\label{eq:zetaL}\Phi_{\cE}(\zeta^!_{\Ku(X_1,\cL)}(L^i_1))\cong \zeta^!_{\Ku(X_2,\cM)}(M^i_1)[t_i]
\end{equation}
for every $1\leq i\leq c$ and some $t_i\in\Z$. Without loss of generality, in the rest of the argument we can assume $t_1=0$. Let us show how to extend $\mathsf{F}$ to an equivalence $\langle\Ku(X_1,\cL),\cL_1\rangle\cong\langle\Ku(X_2,\cM),\cM_1\rangle$. The general argument by induction works literally along the exact same lines.

By Proposition \ref{prop:extension}, there exists $\cE_1\in\Db(X_1\times X_2)$ such that 
\begin{enumerate}
    \item [\rm{(1)}] $\Phi_{\cE_1}|_{\Ku(X_1,\cL)}\cong \mathsf F$ and $\Phi_{\cE_1}(^\perp\langle\Ku(X_1,\cL),L^1_1\rangle)\cong \mathsf 0$;
    \item [\rm{(2)}] $\Phi_{\cE_1}|_{\langle\Ku(X_1,\cL),L^1_1\rangle}\colon\langle\Ku(X_1,\cL),L^1_1\rangle\rightarrow \langle\Ku(X_2,\cM),M^1_1\rangle$ is an equivalence;
    \item [\rm{(3)}] $\Phi_{\cE_1}(L^1_1)\cong M^1_1$.
\end{enumerate}
Actually, it is important to note that the admissible subcategories $\cL':=\langle\cL'_1,\dots,\cL_c\rangle$ and $\cM':=\langle \cM'_1,\dots,\cM_c\rangle$, where $\cL'_1:=\langle L^1_2,\dots, L^1_{n_1}\rangle$ and $\cM'_1:=\langle M^1_2,\dots, M^1_{n'_1}\rangle$, give rise to semiorthogonal decompositions $\Db(X_1)=\langle\Ku(X_1,\cL'),\cL\rangle$ and $\Db(X_2)=\langle\Ku(X_2,\cM'),\cM\rangle$ as in Setup \ref{setup}. Here, $\Ku(X_1,\cL'):=\langle\Ku(X_1,\cL),L^1_1\rangle$ and $\Ku(X_2,\cM'):=\langle\Ku(X_2,\cM),M^1_1\rangle$.

If $n_1=1$, then, by Theorem \ref{thm:sphericals}, $n'_1=1$. Then the first step of the extension is complete. Moreover, $\Phi_{\cE_1}$ satisfies assumptions (a) and (b) of Proposition \ref{prop:extension} for $\cA_1=\langle\Ku(X_1,\cL),L_1^1\rangle$ and $\cA_2=\langle\Ku(X_2,\cM),M_1^1\rangle$. Thus we can proceed further as above.

Assume then $n_1\geq 2$. Since the $\cL_i$'s (and the $\cM_i$'s) are completely orthogonal to each other, we get the isomorphisms $\zeta^!_{\Ku(X_1,\cL)}(L_1^i)\cong\zeta^!_{\Ku(X_1,\cL')}(L_1^i) $ and $\zeta^!_{\Ku(X_2,\cM)}(M_1^i)\cong\zeta^!_{\Ku(X_2,\cM')}(M_1^i)$, for every $i\geq 2$.

By \eqref{eq:zetaL} and  Theorem \ref{thm:sphericals},
\begin{equation}\label{eq:PhizetaL}
    \Phi_{\cE_1}(\zeta^!_{\Ku(X_1,\cL')}(L^1_2))\cong \zeta^!_{\Ku(X_2,\cM')}(M^1_2)[t],
\end{equation}
for some $t\in \Z$. In particular, $n'_1\geq 2$. Moreover, we have the following chain of isomorphisms
\begin{align*}
\RHom(L^1_1,L^1_2)&\cong\RHom(L^1_1,\zeta^!_{\Ku(X_1,\cL')}(L^1_2))\\
    & \cong\RHom(\Phi_{\cE_1}(L^1_1),\Phi_{\cE_1}(\zeta^!_{\Ku(X_1,\cL')}(L^1_2)))\\&\cong\RHom(M^1_1,\zeta^!_{\Ku(X_2,\cM')}(M^1_2)[t])\\
    & \cong \RHom(M^1_1,M^1_2[t])\\
    &\cong\K[t]\oplus \K[t-1],
\end{align*}
where the first and penultimate one follows by adjunction, the second one uses that $\Phi_{\cE_1}|_{\Ku(X_1,\cL')}$ is an equivalence, the third one is by induction and \eqref{eq:PhizetaL} and, finally, the last isomorphism is by the definition of $M_1^1$ and $M^1_2$. Since $\RHom(L^1_1,L^1_2)\cong\K\oplus\K[-1]$, we have $t=0$.

We can then apply Proposition \ref{prop:extension} again in order to extend $\Phi_{\cE_1}|_{\Ku(X_1,\cL')}$ to a Fourier\textendash Mukai functor $\Phi_{\cE_2}$ inducing an equivalence $$\Phi_{\cE_2}:\langle\Ku(X_1,\cL),L^1_1,L^1_2\rangle\stackrel{\sim}{\longrightarrow} \langle\Ku(X_2,\cM),M^1_1,M^1_2\rangle$$ such that
\begin{itemize}
\item $\Phi_{\cE_2}|_{\langle\Ku(X_1,\cL),L^1_1\rangle}\cong\Phi_{\cE_1}|_{\langle\Ku(X_1,\cL),L^1_1\rangle}$,
\item $\Phi_{\cE_2}({^\perp\langle\Ku(X_1,\cL),L^1_1,L^1_2\rangle})\cong \mathsf 0$, and
\item $\Phi_{\cE_2}(L^1_2)\cong M^1_2$.
\end{itemize}
By iterating the procedure, we get a Fourier\textendash Mukai functor $\Phi_{\cE_{n_1}}$ inducing an equivalence
\[
\Phi_{\cE_{n_1}}:\langle\Ku(X_1,\cL),\cL_1\rangle\stackrel{\sim}{\longrightarrow}\langle\Ku(X_2,\cM),M^1_1,M^1_2,\dots,M^1_{n_1}\rangle
\]
such that $\Phi_{\cE_{n_1}}(L^1_j)\cong M^1_j$. By \eqref{eq:zetaL}, \eqref{eq:PhizetaL} and Theorem \ref{thm:sphericals}, the admissible subcategory $\langle M^1_1,\dots,M^1_{n_1}\rangle$ must coincide with $\cM_1$ and thus $n_1=n'_1$.

At each step, the assumptions (a) and (b) of Proposition \ref{prop:extension} are satisfied, so we can proceed further by induction on $c$ and deal with the other components as we mentioned at the beginning of this proof. In conclusion, in a finite number of step, we get an equivalence $\Db(X_1)\cong\Db(X_2)$ and, by Theorem \ref{thm:dertordercat}, an isomorphism $X_1\cong X_2$. This proves (2) in the statement while properties (1) and (3) are automatic by construction.
\end{proof}

Clearly, Theorem \ref{thm:derived_torreli} has a trivial converse. Indeed, assume we are given an isomorphism $f\colon X_1\to X_2$ and a semiorthogonal decomposition for $\Db(X_2)$ as in Setup \ref{setup}. Since the exact functor $f^*\colon\Db(X_2)\to\Db(X_1)$ is an equivalence, we can take on $\Db(X_1)$ the semiorthogonal decomposition as in Setup \ref{setup} which is the image of the given one on $\Db(X_2)$ under $f^*$. The exact functor, $\mathsf{F}:=(f^*)^{-1}|_{\Ku(X_1,\cL_1)}\colon\Ku(X_1,\cL_1)\xrightarrow{\sim}\Ku(X_2,\cL_2)$ is an exact equivalence of Fourier\textendash Mukai type by construction. Of course, the two semiorthogonal decompositions on $X_1$ and $X_2$ have the same type and, in the argument, we can exchange the roles of $X_1$ and $X_2$.


\bigskip

{\small\noindent{\bf Acknowledgements.} We are greatly in debt with Alexander Kuznetsov whose comments on our previous paper \cite{LNSZ} pushed us to refine our classification of special objects in the Kuznetsov component in order to deal with non-generic Enriques surfaces. We are also very grateful to Igor Dolgachev for patiently answering our questions about Enriques surfaces and to Andreas Hochenegger who carefully read a preliminary version of this paper and gave us several insightful suggestions which improved the presentation. Finally, we would like to thank Marcello Bernardara, Daniele Faenzi, Sukhendu Mehrotra, Howard Nuer and Franco Rota who were part of the working group at the \emph{Workshop on  ``Semiorthogonal decompositions, stability conditions and sheaves of categories''} (Toulouse, 2018) where the first part \cite{LNSZ} of this project started.}


\end{document}